\documentclass[12pt]{amsart}

\usepackage[latin1]{inputenc}
\usepackage{amssymb,amsmath,amsfonts, amsthm}
\usepackage{enumerate}
\usepackage{times}
\usepackage[colorlinks, citecolor=blue, linkcolor=red]{hyperref}

\usepackage{bbding}
\usepackage[margin=2.0cm]{geometry}
\usepackage{cite}

\usepackage{hyperref}

\numberwithin{equation}{section}

\theoremstyle{definition}
\newtheorem{theorem}{Theorem}[section]
\newtheorem{lemma}[theorem]{Lemma}

\theoremstyle{remark}

\newcommand{\R}{{\mathbb R}}

\begin{document}

\title[Extremals and critical points of the anisotropic Sobolev inequality in convex cones]{An overview on extremals and critical points of the Sobolev inequality in convex cones}
\author{Alberto Roncoroni}
\address{A. Roncoroni.  Dipartimento di Matematica, Politecnico di Milano, Piazza Leonardo da Vinci 32, 20133, Milano, Italy}
\email{alberto.roncoroni@polimi.it}

\maketitle

\begin{abstract}
In this survey, we consider the sharp Sobolev inequality in convex cones. We also prove it by using the optimal transport technique.  Then we present some results related to the Euler-Lagrange equation of the Sobolev inequaliy: the so-called critical $p-$Laplace equation. Finally, we discuss some stability results related to the Sobolev inequality.
\end{abstract}

\bigskip







\section{Introduction}

The starting point is the well-known Sobolev inequality in $\mathbb{R}^n$ firstly showed in \cite{Sobolev} (see also \cite{Gagliardo,Nirenberg} for alternative proofs and \cite[Theorem IX.9]{Brezisbook}, \cite[Theorem 7.10]{GT} and \cite[Theorem 2.4.1]{Z} for further references): given $n\geq 3$ and $1<p<n$, there exists a positive constant $S=S(n,p)$ such that 
\begin{equation}\label{Sobolev1}
\Vert\nabla u \Vert_{L^p(\mathbb{R}^n)}\geq S\Vert u \Vert_{L^{p^\ast}(\mathbb{R}^n)}\, ,  \quad \text{ for all $u\in W^{1,p}(\mathbb{R}^n)$}\, ,
\end{equation}
where 
$$
p^{\ast}:=\frac{np}{n-p}\, , 
$$ 
is the so-called \emph{Sobolev critical exponent} and the space $W^{1,p}(\mathbb{R}^n)$ is the usual Sobolev space of functions $u\in L^{p}(\mathbb{R}^n)$ such that $\nabla u\in L^p(\mathbb{R}^n)$. Actually, thanks to a short interpolation argument (see e.g. \cite[Corollary IX.10]{Brezisbook}) it is immediate to see that the following version of the Sobolev inequality 
\begin{equation}\label{Sobolev}
\Vert\nabla u \Vert_{L^p(\mathbb{R}^n)}\geq S\Vert u \Vert_{L^{p^\ast}(\mathbb{R}^n)}\, ,  \quad \text{ for all $u\in \dot W^{1,p}(\mathbb{R}^n)$}\, ,
\end{equation}
holds, where 
$$
\dot{W}^{1,p}(\mathbb{R}^n):=\lbrace u\in L^{p^\ast}(\mathbb{R}^n) \, : \, \nabla u\in L^p(\mathbb{R}^n)\rbrace\,  ,
$$ 
is an homogeneous Sobolev space. It is homogeneous in the same sense the Sobolev inequality \eqref{Sobolev} is homogeneous under rescaling $f_{\lambda}(\cdot):=f(\cdot/\lambda)$. As we will see the homogeneous Sobolev space $\dot{W}^{1,p}(\mathbb{R}^n)$ is better than the Sobolev space $W^{1,p}(\mathbb{R}^n)$ in order to study extremals of the Sobolev inequality.

An interesting and fascinating aspect is to prove the sharp version of the Sobolev inequality \eqref{Sobolev}, this means that one wants to characterize the extremals of \eqref{Sobolev}, i.e. functions that realize the equality in \eqref{Sobolev}, and to compute the best constant in \eqref{Sobolev}. 
This has been done independently in two contemporary papers: \cite{Aubin} and \cite{Talenti} by using symmetrizations. In particular, they show that the extremals of \eqref{Sobolev} can be explicitly computed and are of the following form: 
\begin{equation}\label{bubbles}
\mathcal{U}_{a,\lambda,x_0}(x):=\dfrac{a}{\left(1+ \lambda^{\frac{p}{p-1}}\vert x-x_0\vert^{\frac{p}{p-1}} \right)^{\frac{n-p}{p}}}\, , 
\end{equation}
where $a\in\mathbb{R}\, , \lambda>0$ and $ x_0\in\mathbb{R}^n$ can be chosen arbitrary; moreover, the best constant in \eqref{Sobolev} is given by
\begin{equation}\label{opt_constant}
\sqrt{\pi}n^{\frac{1}{p}}\left(\frac{n-p}{p-1}\right)^{\frac{p-1}{p}}\left(\frac{\Gamma(n/p)\Gamma(1+n-n/p)}{\Gamma(1+n/2)\Gamma(n)}\right)^{\frac{1}{n}}\, .
\end{equation}
In this paper we focus on the different approach, proposed in \cite{CNV}, to prove the sharp Sobolev inequality (see Section \ref{OT} for more details). The approach in \cite{CNV} is based on the optimal transport technique and, actually, this approach is suitable to be adapted to show also the sharp Sobolev inequality in $\mathbb{R}^n$ in the anisotropic setting, i.e. in $\mathbb{R}^n$ endowed with a generic anisotropic norm. 
In this context the (anisotropic) Sobolev inequality becomes the following: given $n\geq 3$ and $1<p<n$ there exists a positive constant $S=S(n,p,H)$ such that
 \begin{equation}\label{Sobolev_anis}
\Vert H(\nabla u) \Vert_{L^p(\mathbb{R}^n)}\geq S\Vert u \Vert_{L^{p^\ast}(\mathbb{R}^n)}\, ,  \quad \text{ for all $u\in \dot{W}^{1,p}(\mathbb{R}^n)$}\, ,
\end{equation}
where $H:\mathbb{R}^n\rightarrow\mathbb{R}$ is an anisotropic norm (or gauge), i.e. 
\begin{center}
$H$ is positive, positively homogeneous of degree one\footnote{$
H(\lambda\xi)=\lambda H(\xi)$, for all $\lambda>0$ and $\xi\in\mathbb{R}^n$. Note that, in general, we do not require $H$ to be symmetric, so it may happen that $H(\xi)\neq H(-\xi)$.} and convex\footnote{We mention two typical classes of anisotropic norms: the first ones are the so-called $\ell_p$-norms given by
$$
H(x)=\Vert x \Vert_p:=\left( \vert x_1\vert^p + \dots + \vert x_n\vert^p\right)^{1/p}\,, \quad \text{ for $1< p<\infty$ and $x\in\mathbb{R}^n$}\, ;
$$ 
the second ones are the so-called crystalline norms: given a finite set $\lbrace p_j\rbrace_{j=1}^{N}\subset\mathbb{R}^n\setminus\lbrace 0\rbrace$, with $N\in\mathbb{N}$, consider 
$$
H(x):=\max_{1\leq j\leq N} x\cdot p_j\,, \quad \text{ for $x\in\mathbb{S}^{n-1}$}\, .
$$
};
\end{center}
and 
$$
\dot{W}^{1,p}(\mathbb{R}^n):=\left\lbrace u\in L^{p^\ast}(\mathbb{R}^n) \, : \, \int_{\mathbb{R}^n}H(\nabla u)^p\, dx<+\infty \right\rbrace\,  .
$$
Moreover, thanks to \cite{CNV} we know that the extremals of \eqref{Sobolev_anis} are of the following form:
$$
\mathcal{U}^H_{a,\lambda,x_0}(x):=\frac{a}{\left( 1 + \lambda^\frac{p}{p-1} H_0(x_0-x)^\frac{p}{p-1} \right)^{\frac{n-p}{p}}} \, , 
$$
where $a\in\mathbb{R}$, $ \lambda>0$ and $x_0\in\mathbb{R}^n$ can be chosen arbitrary and where $H_0$ is the dual norm of $H$, explicitly 
$$
H_0(\zeta):=\sup_{H(\xi)=1}\zeta\cdot\xi \, , \quad \text{for all $\zeta\in\mathbb{R}^n$}\, .
$$
Another interesting aspect related to the Sobolev inequality is about its validity in convex cones of $\mathbb{R}^n$. We recall that an open cone $\Sigma\subset\mathbb{R}^n$ is given by
$$
\Sigma:=\lbrace tx \, : \,  t\in(0,+\infty) \, , \, x\in\omega\rbrace \, ,
$$
for some open domain $\omega\subseteq\mathbb{S}^{n-1}$.  In particular, we are interested in convex cones and so we recall that every convex cone $\Sigma$ of $\mathbb{R}^n$ can be decomposed in the following way: 
$$
\Sigma=\mathbb{R}^{k}\times \mathcal{C} \, ,
$$
for some $k\in\lbrace 0,\dots,n\rbrace$ and $\mathcal{C}\subset\mathbb{R}^{n-k}$ is a convex cone that does not contain a line and with only one vertex $\mathcal{O}$ (from now on and for simplicity we will assume that $\mathcal{O}$ coincides with the origin). The Sobolev inequality in convex cones has been firstly established in \cite{LPT} and it has been generalized to the anisotropic setting in \cite{CROS} where it takes the following form: given $n\geq 3$ and $1<p<n$, there exists a positive constant $S_\Sigma=S_\Sigma(n,p,H)$ such that
\begin{equation}\label{Sobolev_cone}
\left(\int_{\Sigma} H(\nabla u)^p \, dx\right)^{\frac{1}{p}}\geq S_\Sigma\left(\int_{\Sigma} \vert u(x)\vert^{p^\ast} \, dx\right)^{\frac{1}{p^\ast}}\, ,  \quad \text{ for all $u\in \dot{W}^{1,p}(\Sigma)$}\, ;
\end{equation}
where $\Sigma\subset\mathbb{R}^n$ is a convex cone, $H$ is an anisotropic norm in $\mathbb{R}^n$ and 
$$
\dot{W}^{1,p}(\Sigma):=\left\lbrace u\in L^{p^\ast}(\Sigma) \, : \, \int_{\Sigma}H(\nabla u)^p\, dx<+\infty \right\rbrace\,  .
$$
The point is that in \cite{CROS} the characterization of extremals of \eqref{Sobolev_cone} (i.e. the sharp anisotropic Sobolev inequality in convex cones) is missing, indeed they deduce the anisotropic Sobolev inequality in convex cones as a corollary of the anisotropic isoperimetric inequality (the Wulff inequality, see Subsection \ref{isop} for more details) in convex cones.  Actually, in \cite{LPT} the sharp Sobolev inequality can be found in the case $p=2$ and $H(\cdot)=\vert\cdot\vert$. Exploiting the optimal transport approach, in \cite{CFR}, we prove the sharp anisotropic Sobolev inequality in convex cones computing the extremals and we show that they are of the following form: 
$$
\mathcal{U}_{a,\lambda,x_0}^H(x):=\frac{a}{\left( 1 + \lambda^\frac{p}{p-1} H_0(x_0-x)^\frac{p}{p-1} \right)^{\frac{n-p}{p}}} \, , 
$$
where $a\in\mathbb{R}$ and $ \lambda>0$ can be chosen arbitrary and where, as before,  $H_0$ is the dual norm of $H$. Moreover, if $\Sigma=\mathbb{R}^n$ then $x_0$ may be any point of $\mathbb{R}^n$, if $\Sigma=\mathbb{R}^k\times \mathcal{C}$ with $k\in  \lbrace\, 1,\dots,n-1\rbrace$ and $\mathcal{C}$ does not contain a line, then $x_0\in \mathbb{R}^k \times \lbrace \mathcal{O}\rbrace$, otherwise $x_0 = \mathcal{O}$. 

\subsection{Analogies with the isoperimetric inequality.}\label{isop} We conclude this introduction by mentioning that the same picture drown for the Sobolev inequality holds also for the isoperimetric inequality.  The classical Euclidean isoperimetric inequality states that, for any bounded open (smooth) set $E\subset\mathbb{R}^n$, the perimeter $P(E)$ controls the volume $\vert E\vert$: more precisely,
\begin{equation}\label{iso}
P(E)\geq n \vert B_1\vert^{\frac{1}{n}} \vert E\vert^{\frac{n-1}{n}} \, , 
\end{equation}
where $B_1$ is the unit ball in $\mathbb{R}^n$.  Moreover, the equality holds if and only if $E$ is a ball (see \cite{DeGiorgi} and also \cite{Osserman}). A way to define the perimeter (in the smooth setting) is the following: 
$$
P(E):=\int_{\partial E}\, d\mathcal{H}^{n-1} \, ,
$$
where $\mathcal{H}^{n-1}$ denotes the $(n-1)-$dimensional Hausdorff measure in $\mathbb{R}^n$ (for the general definition of perimeter we refer to the book \cite{Maggi}).  More in general we can consider the notion of anisotropic perimeter, i.e. the perimeter induced by an anisotropic norm $H$ in $\mathbb{R}^n$, i.e. 
$$
P_H(E):=\int_{\partial E}H(\nu(x)) \, d\mathcal{H}^{n-1}(x) \, ,
$$
where $\nu(x)$ denotes the unit outward normal at $x\in\partial E$. The corresponding anisotropic isoperimetric inequality, known as the Wullf's inequality is the following:
\begin{equation}\label{iso_anis}
P_H(E)\geq n \vert K\vert^{\frac{1}{n}} \vert E\vert^{\frac{n-1}{n}} \, , 
\end{equation}
where $K$ is the Wulff shape (or Alexandrov's body) associated to $H$, explicitly
$$
K:=\lbrace x\in\mathbb{R}^n \, :\, x\cdot\nu<H(\nu)\, , \, \text{ for all $\nu\in\mathbb{S}^{n-1}$}\rbrace\, .
$$
In addition, the equality holds if and only if $E = x + \rho K$ for some $x\in\mathbb{R}^n$ and $\rho>0$ (see \cite{Taylor1,Taylor2,Wulff} and also \cite{Fonseca}).  In the context of a convex cone $\Sigma$ of $\mathbb{R}^n$ the definition of perimeter becomes the following: 
$$
P(E,\Sigma):=\int_{\partial E\cap\Sigma}\, d\mathcal{H}^{n-1} \, ,
$$
i.e. we consider the relative perimeter of the set $E$ with respect to $\Sigma$. The analogue of \eqref{iso} in a convex cone $\Sigma$ of $\mathbb{R}^n$ is the following isoperimetric inequality: for every bounded open (smooth) set $E\subset\Sigma$ we have
$$
P(E,\Sigma)\geq n \vert B_1\cap\Sigma\vert^{\frac{1}{n}} \vert E\vert^{\frac{n-1}{n}} \, .
$$
Furthermore, if the cone contains no lines then the equality holds if and only if $E$ is a spherical sector centred at the vertex of the cone (see \cite{LionsPacella} and also \cite{APPS,RR,FI,CROS,BF,BCM,Indrei} for generalisations and different proofs). In particular, in \cite{CROS} and in \cite{FI} the authors prove the following Wulff inequality in a convex cone $\Sigma$:
\begin{equation}\label{iso_anis_cone}
P_H(E,\Sigma)\geq n \vert K\cap\Sigma\vert^{\frac{1}{n}} \vert E\vert^{\frac{n-1}{n}} \, ,
\end{equation}
where $P_H(E,\Sigma)$ denotes the anisotropic relative perimeter of the set $E$ with respect to $\Sigma$, i.e.
$$
P_H(E,\Sigma):=\int_{\partial E\cap\Sigma}H(\nu(x)) \, d\mathcal{H}^{n-1}(x) \,.
$$
Moreover, if the cone contains no lines then the equality holds in \eqref{iso_anis_cone} if and only if $E=\rho K$, for some $\rho>0$.

\medskip

\subsection*{Organization of the paper} The paper is organized as follows: in Section \ref{OT} we show how to apply the optimal transport approach to prove the sharp Sobolev inequality in a convex cone $\Sigma$ of $\mathbb{R}^n$, in Section \ref{FR} we investigate another important aspect related to the sharp Sobolev inequality, i.e. the characterization of critical points. Finally, in Section \ref{quant} we present some stability results related to the extremals and to the critical points of the Sobolev inequality.


\section{Proof by optimal transport}\label{OT}

The idea to apply the optimal transport theory to prove functional and geometric inequalities dates back to Gromov (see the appendix in the book \cite{Gromov}), who applied the Knothe map\footnote{in his original argument, Gromov did not use the optimal transport map but instead the Knothe map, see e.g. \cite[Section 1.4]{FMP} for more details.} to prove the isoperimetric inequality \eqref{iso}. This approach has been brilliantly used in \cite{FMP} to show the so-called Wulff inequality, i.e the anisotropic isoperimetric inequality \eqref{iso_anis}.  We mention also that in \cite{FI} the optimal transport approach is used to prove the anisotropic Sobolev inequality in a convex cone \eqref{iso_anis_cone}. 

Finally, the optimal transport approach has been also used to prove the Brunn-Minkowski inequality (see e.g.  \cite{MCann}) and to prove the anisotropic Gagliardo--Nirenberg inequalities (see \cite{CNV}). 

Here we sketch the proof of the sharp Sobolev inequality in a convex cone $\Sigma$ of $\mathbb{R}^n$ in the isotropic case (i.e. when $H$ is the Eulidean norm) and we refer to \cite[Sections 2 and 4]{CNV} and to \cite[Appendix A]{CFR} for the proof in the anisotropic setting in the case $\Sigma\equiv\mathbb{R}^n$ and in the conical case, respectively. In particular, we prove the following 

\begin{theorem}[Sharp Sobolev inequality in convex cones]\label{Sob_sharp}
Given $n\geq 3$ and $1<p<n$ let $\Sigma$ be a convex cone of $\mathbb{R}^n$. Then there exists a positive constant $S_{\Sigma}=S_{\Sigma}(n,p)$ such that 
\begin{equation}\label{Sob_proof}
\left(\int_{\Sigma} \vert\nabla u(x)\vert^p\, dx\right)^{\frac{1}{p}}\geq S_{\Sigma} \left(\int_{\Sigma} \vert u(x)\vert^{p^\ast}\, dx\right)^{\frac{1}{p^\ast}} \, ,  \quad \text{ for all $u\in \dot{W}^{1,p}(\Sigma)$}\, .
\end{equation}
Moreover, the inequality is sharp and the equality in \eqref{Sob_proof} is attained if and only if $u(x)=\mathcal{U}_p(x)$, where 
$$
\mathcal{U}_{a,\lambda,x_0}(x):=\frac{a}{\left( 1 + \lambda^\frac{p}{p-1} \vert x-x_0\vert)^\frac{p}{p-1} \right)^{\frac{n-p}{p}}} \, , 
$$
with $a\in\mathbb{R}$, $ \lambda>0$ and $x_0\in\overline{\Sigma}$. Furthermore,  writing $\Sigma = \mathbb{R}^k\times\mathcal{C}$ with $k \in \lbrace 0, \dots, n\rbrace$ and with $\mathcal{C}\subset \mathbb{R}^{n-k}$ a convex cone that does not contain a line and with only one vertex $\mathcal{O}$, then 
\begin{itemize}
	\item[$(i)$] if $k=n$ then $\Sigma = \mathbb{R}^n$ and $x_0$ may be a generic point in $\mathbb{R}^n$;
	\item[$(ii)$] if $k\in\{1,\dots,n-1\}$ then $x_0\in\mathbb{R}^k\times\mathcal{\{\mathcal O\}}$;
	\item[$(iii)$] if $k=0$ then $x_0=\mathcal{O}$.
\end{itemize} 
\end{theorem}

As already mentioned a key ingredient in the proof of Theorem \ref{Sob_sharp} is the optimal transport theory (see \cite{AmGi,San,Villani_bis,Villani} for general references); in the next subsection we recall briefly some basic tools from optimal transport theory and we recall the main theorem that we are going to use. In what follows $\Sigma$ will be a convex cone of $\mathbb{R}^n$, even if some results are true more in general.

\subsection{Optimal Transport Theory}

Given two probability densities $F$ and $G$ on $\Sigma$ (i.e. two nonnegative functions in $L^1(\Sigma)$ such that $\Vert F\Vert_{L^1(\Sigma)}=\Vert G\Vert_{L^1(\Sigma)}=1$),  we say that a map $\mathcal{T}:\Sigma\rightarrow\Sigma$ sends $F$ onto $G$, and we will call it a \emph{transport map}, if 
$$
\mathcal{T}_\# F=G\, , 
$$
i.e. the push forward of $F$ through $\mathcal{T}$ is $G$. Explicitly, if
\begin{equation}\label{trans}
\int_E G(y)\, dy=\int_{\mathcal{T}^{-1}(E)} F(x)\, dx \, , \quad \text{for every $E$ Borel subset of $\Sigma$;}
\end{equation}
or equivalently,  if
\begin{equation}\label{trans_bis}
\int_{\Sigma} b(y)G(y)\, dy=\int_{{\Sigma}}b(\mathcal{T}(x)) F(x)\, dx \, , \quad \text{for every $b:\Sigma\rightarrow\mathbb{R}$ nonnegative Borel function.}
\end{equation}
Moreover, we consider the following cost function $c:\Sigma\times\Sigma\rightarrow\mathbb{R}$ given by $c(x,y)=\frac{\vert x-y\vert^2}{2}$ and we define the \emph{total cost of a transport map} $\mathcal{T}$ as
$$
\mathrm{cost}(\mathcal{T}):=\int_{\Sigma}c(x,\mathcal{T}(x))F(x)\, dx\, . 
$$
In general, we are interested in the \emph{optimal transport map}, i.e. the transport map $\mathcal{T}$ that minimizes the total cost, i.e. 
$$
\mathrm{cost}(\mathcal{T})=\min\lbrace \mathrm{cost}(\mathcal{S}) \, : \,  \mathcal{S}:\Sigma\rightarrow\Sigma  \, , \,  \mathcal{S}_\# F=G \rbrace\, . 
$$
The central ingredient is the following existence theorem for the optimal transport map (see e.g.  \cite{DeFi}).

\begin{theorem}\label{Brenier}
If $F$ and $G$ are two probability densities on $\Sigma$,  then there exist a convex function $\varphi:\Sigma\rightarrow\mathbb{R}$ such that the transport map $\mathcal{T}:\Sigma\rightarrow\Sigma$, defined by 
$$
\mathcal{T}(x):=\nabla\varphi(x)\, ,
$$
is the unique optimal transport map that sends $F$ onto $G$.  We will refer to $\mathcal{T}$ as the \emph{Brenier map}. Moreover,  $\mathcal{T}$ is differentiable $F(x)\, dx-$a.e. and 
\begin{equation}\label{MA}
\vert \mathrm{det}(\nabla \mathcal{T}(x))\vert=\frac{F(x)}{G(\mathcal{T}(x))} \, , \quad \text{ $F(x)\, dx-$a.e.  $x\in\Sigma$.}
\end{equation}
\end{theorem}

Observe that if $\mathcal{T}$ is a diffeomorphism, then the change of variables $y=\mathcal{T}(x)$ in \eqref{trans} and in \eqref{trans_bis} shows that $\mathcal{T}$ solves \eqref{MA}, which in terms of $\varphi$ becomes the following Monge-Ampère equation:
\begin{equation}\label{MA_bis}
\vert \mathrm{det}(\nabla^2 \varphi (x))\vert=\frac{F(x)}{G(\nabla\varphi(x))} \, .
\end{equation}
More in general,  \eqref{MA} and \eqref{MA_bis} hold for $F(x)\, dx-$a.e.  $x\in\Sigma$ without further assumptions on $F$ and $G$ below integrability (see e.g.  \cite[Chapter 10]{Villani_bis} and \cite[Theorem 3.6]{DeFi}). In this case, being $\varphi$ a convex function it admits, almost everywhere, a second-order Taylor expansion as shown in \cite[Chapter 6]{EG}.

\subsection{Proof of Theorem \ref{Sob_sharp}}

The first step in the proof of Theorem \ref{Sob_sharp} is the following 

\begin{lemma}\label{key lemma}
Given $n\geq 3$, $1<p<n$ and $q=p/(p-1)$.  Let $\Sigma$ be a convex cone of $\mathbb{R}^n$.  Whenever $f\in \dot{W}^{1,p}(\Sigma)$ and $g\in L^{p^\ast}(\Sigma)$ are two functions with $\Vert f\Vert_{L^{p^\ast}(\Sigma)}=\Vert g\Vert_{L^{p^\ast}(\Sigma)}$, then 
\begin{equation}\label{formula}
\int_{\Sigma}\vert g(x)\vert^{p^\ast\left(1-\frac{1}{n}\right)}\, dx \leq \frac{p(n-1)}{n(n-p)} \Vert \nabla f\Vert_{L^p(\Sigma)} \left(\int_{\Sigma}\vert y\vert^q \vert g(y)\vert^{p^\ast}\, dy\right)^{\frac{1}{q}}\, , 
\end{equation}
with equality if 
$$
f(x)=g(x)=\mathcal{U}_{1,1,\mathcal{O}}(x)=\frac{1}{\left( 1 +  \vert x\vert)^\frac{p}{p-1} \right)^{\frac{n-p}{p}}}\, ,
$$
recall that $\mathcal{O}$ coincides with the origin.
\end{lemma}

\begin{proof}[Proof of Lemma \ref{key lemma}]
First of all, it is well-known that whenever $f\in\dot{W}^{1,p}(\Sigma)$ then $\nabla\vert f\vert=\pm\nabla f$ almost everywhere, so $f$ and $\vert f\vert$ have equal Sobolev norms. Thus, without loss of generality, we may assume that $f$ and $g$ are nonnegative and, by homogeneity, satisfy
$$
\Vert f\Vert_{L^{p^{\ast}}(\Sigma)}=1=\Vert g\Vert_{L^{p^{\ast}}(\Sigma)}\, . 
$$
Secondly, we prove \eqref{formula} in the special case when $f$ and $g$ are smooth functions with compact support inside $\overline{\Sigma}$; the general case will follow by approximation.

Now, consider the two probability densities 
$$
F(x)=f^{p^{\ast}}(x) \quad \text{and} \quad G(x)=g^{p^{\ast}}(x)\, . 
$$
on $\Sigma$.  Let $\mathcal{T}:\Sigma\rightarrow\Sigma$ be the optimal transport map given by Theorem \ref{Brenier}. Thanks to the regularity theory developed in \cite{CoFi} we know that $\mathcal{T}$ is a diffeomorphism and it satisfies 
\begin{equation}\label{MA_ter}
\vert \mathrm{det}(\nabla \mathcal{T}(x))\vert=\frac{f^{p^\ast}(x)}{g^{p^\ast}(\mathcal{T}(x))} \, , \quad \text{ for all $x\in\Sigma$.}
\end{equation}
From the definition of $F$ and $G$ and from the transport condition \eqref{trans_bis}, with $b(y)=G^{-\frac{1}{n}}(y)$ , we know that 
\begin{multline}\label{A.4}
\int_{\Sigma} g^{p^\ast\left(1-\frac{1}{n}\right)}(y) \, dy=\int_{\Sigma} G^{1-\frac{1}{n}}(y) \, dy=\int_{\Sigma} G^{-\frac{1}{n}}(y) G(y) \, dy\\ = \int_{\Sigma} G^{-\frac{1}{n}}(\mathcal{T}(x)) F(x)\, dx =\int_{\Sigma} g^{-\frac{p^\ast}{n}}(\mathcal{T}(x)) f^{p^\ast}(x)\, dx \, .
\end{multline}
While, from \eqref{MA_ter} we get 
\begin{equation}\label{A.5}
\int_{\Sigma} g^{-\frac{p^\ast}{n}}(\mathcal{T}(x)) f^{p^\ast}(x)\, dx =\int_{\Sigma} \vert\mathrm{det}(\nabla\mathcal{T}(x))\vert^\frac{1}{n}f^{p^\ast\left(1-\frac{1}{n}\right)}(x) \, dx\, . 
\end{equation}
Since $\mathcal{T}=\nabla\varphi$, for some convex function $\varphi:\Sigma\rightarrow\mathbb{R}$, then $\nabla \mathcal T=\nabla^2\varphi$ is symmetric and non-negative definite. In particular $\mathrm{det}(\nabla\mathcal{T})\geq 0$ and from the arithmetic-geometric inequality we get 
\begin{equation}\label{AMGM}
\vert\mathrm{det}(\nabla\mathcal{T})\vert^\frac{1}{n}\leq \frac{1}{n} 	\mathrm{div}(\mathcal{T})\, . 
\end{equation}
Hence, combining \eqref{AMGM} with \eqref{A.4} and \eqref{A.5}, 
\begin{multline}
\int_{\Sigma} g^{p^\ast\left(1-\frac{1}{n}\right)}(y) \, dy\leq \frac{1}{n} \int_{\Sigma} \mathrm{div}(\mathcal{T}(x))f^{p^\ast\left(1-\frac{1}{n}\right)}(x) \, dx\\ = -\dfrac{p^\ast}{n}\left(1-\dfrac{1}{n}\right)\int_{\Sigma} f^{p^\ast\left(1-\frac{1}{n}\right)-1}(x) \mathcal{T}(x)\cdot\nabla f(x)\, dx +\frac{1}{n}\int_{\partial\Sigma} f^{p^\ast\left(1-\frac{1}{n}\right)}(x) T(x)\cdot\nu(x)\, d\sigma(x)\, . 
\end{multline}
where we have used the integration by parts formula. Now, observe that since $\mathcal{T}(x)\in\overline{\Sigma}$, for all $x\in\overline{\Sigma}$, the convexity of $\Sigma$ implies that 
\begin{equation}\label{normal}
\mathcal{T}(x)\cdot\nu(x)\leq 0\, , \quad \text{ for all $x\in\partial\Sigma$}\, .
\end{equation}
Thus, 
\begin{equation}\label{dddd}
\int_{\Sigma} g^{p^\ast\left(1-\frac{1}{n}\right)}(y) \, dy\leq -\dfrac{p^\ast}{n}\left(1-\dfrac{1}{n}\right)\int_{\Sigma} f^{p^\ast\left(1-\frac{1}{n}\right)-1}(x) \mathcal{T}(x)\cdot\nabla f(x)\, dx \, , 
\end{equation}
and from Hölder's inequality we conclude that 
\begin{multline}\label{Holder}
-\int_{\Sigma} f^{p^\ast\left(1-\frac{1}{n}\right)-1}(x) \mathcal{T}(x)\cdot\nabla f(x)\, dx \leq \Vert\nabla f\Vert_{L^p(\Sigma)} \left( \int_{\Sigma} \vert\mathcal{T}(x)\vert^{q} f^{p^\ast}(x)\, dx\right)^{\frac{1}{q}} \\
=\Vert\nabla f\Vert_{L^p(\Sigma)} \left( \int_{\Sigma} \vert y\vert^{q} g^{p^\ast}(y)\, dy\right)^{\frac{1}{q}}\, ,
\end{multline}
hence 
$$
\int_{\Sigma} g^{p^\ast\left(1-\frac{1}{n}\right)}(y) \, dy\leq \dfrac{p^\ast}{n}\left(1-\dfrac{1}{n}\right)\Vert\nabla f\Vert_{L^p(\Sigma)} \left( \int_{\Sigma} \vert y\vert^{q} g^{p^\ast}(y)\, dy\right)^{\frac{1}{q}}\,, 
$$
which is \eqref{formula} since 
$$
\dfrac{p^\ast}{n}\left(1-\dfrac{1}{n}\right)=\dfrac{p(n-1)}{n(n-p)}\, . 
$$
In the special case $f=g=\mathcal{U}_{1,1,\mathcal{O}}$, the Brenier map coincides with the identity map, i.e. $\mathcal{T}(x)=x$ and $\mathrm{det}(\nabla\mathcal{T})=1$.  This implies that the inequalities in \eqref{AMGM} and in \eqref{normal} become equalities; in particular also in \eqref{dddd} the inequality becomes equality. Moreover from a direct computation one can show that
$$
-\int_{\Sigma}\mathcal{U}_{1,1,\mathcal{O}}^{\frac{p^\ast}{q}}(x)\nabla\mathcal{U}_{1,1,\mathcal{O}}(x)\cdot x\, dx=\Vert\nabla\mathcal{U}_{1,1,\mathcal{O}}\Vert_{L^p(\Sigma)}\left(\int_{\Sigma}\vert x\vert^q \mathcal{U}_{1,1,\mathcal{O}}^{p^\ast}(x)\, dx\right)^{\frac{1}{q}}\, , 
$$
which ensures that also in \eqref{Holder} the equality holds. This is the end of the proof of the Lemma.
\end{proof}

We are now in position to prove Theorem \ref{Sob_sharp}

\begin{proof}[Proof of Theorem \ref{Sob_sharp}]
An immediate consequence of Lemma \ref{key lemma} is the following duality principle:
\begin{equation}\label{duality}
\sup_{\Vert g\Vert_{L^{p^\ast}(\Sigma)}=1}\frac{\int_{\Sigma}\vert g(x)\vert^{p^\ast\left(1-\frac{1}{n}\right)}\, dx}{\int_{\Sigma}\vert y\vert^q \vert g(y)\vert^{p^\ast}\, dy} = \frac{p(n-1)}{n(n-p)}\inf_{\Vert f\Vert_{L^{p^\ast}(\Sigma)}=1} \Vert \nabla f\Vert_{L^p(\Sigma)}\, . 
\end{equation}
with $\mathcal{U}_{1,1,\mathcal{O}}$ extremal in both variational problems. 

From \eqref{duality} the proof of \eqref{Sob_proof} is immediate, indeed let $u\in \dot W^{1,p}(\Sigma)$ and let $\mathcal{U}_{1,1,\mathcal{O}}$ be the extremal, then from \eqref{duality} we have 
$$
\dfrac{\Vert \nabla \mathcal{U}_{1,1,\mathcal{O}}\Vert_{L^p(\Sigma)}}{\Vert \mathcal{U}_{1,1,\mathcal{O}}\Vert_{L^{p^\ast}(\Sigma)}}=\inf_{\Vert f\Vert_{L^{p^\ast}(\Sigma)}=1} \Vert \nabla f\Vert_{L^p(\Sigma)}\leq
\dfrac{\Vert \nabla u\Vert_{L^p(\Sigma)}}{\Vert u\Vert_{L^{p^\ast}(\Sigma)}}
$$
thus, 
$$
\Vert \nabla u\Vert_{L^p(\Sigma)}\geq \dfrac{\Vert \nabla \mathcal{U}_{1,1,\mathcal{O}}\Vert_{L^p(\Sigma)}}{\Vert \mathcal{U}_{1,1,\mathcal{O}}\Vert_{L^{p^\ast}(\Sigma)}} \Vert u\Vert_{L^{p^\ast}(\Sigma)}
$$
which is \eqref{Sob_sharp}. The fact that the inequality is sharp follows from the fact that $u=\mathcal{U}_{1,1,\mathcal{O}}$ realizes the equality.

Now we deal with the characterization of extremals.  We want to prove the following: 
\begin{center}
A function $u\in \dot{W}^{1,p}(\Sigma)$ realizes the equality in the Sobolev inequality \eqref{Sob_proof} if and only if there exist $a\in\mathbb{R}$, $\lambda>0$ and $x_0\in\overline{\Sigma}$ as in $(i)-(ii)-(iii)$ such that
$$
u(x)=a\, \mathcal{U}_{1,1,\mathcal{O}}(\lambda(x-x_0))=\mathcal{U}_{a,\lambda,x_0}(x)\, .
$$
\end{center}

To prove this fact we follow the approach in \cite[Appendix A]{CFR}, in \cite[Section 4]{CNV} and in \cite[Appendix A]{FMP}. Firstly we observe that if $u$ is an extremal, then also $\vert u\vert$ will be an extremal and then the conclusion of the theorem will force $u$ to have constant sign. Hence, it is enough to consider nonnegative functions $u\in \dot{W}^{1,p}(\Sigma)$.  Secondly, given two nonnegative measurable functions $f$ and $g$ such that $\int_{\Sigma}f^{p^\ast}=\int_{\Sigma}g^{p^\ast}$ then saying that $f(x)=a\, g(\lambda(x-x_0)$, for $a\in\mathbb{R}$, $\lambda>0$ and $x_0\in\overline{\Sigma}$,  is equivalent to say that the Brenier map $\mathcal{T}:\Sigma\rightarrow\Sigma$ which sends $f^{p^\ast}$ onto $g^{p^\ast}$ is of the form $\mathcal{T}(x)=\tilde\lambda(x-x_0)$, for some $\tilde\lambda>0$ and $x_0\in\overline{\Sigma}$. Furthermore, we can assume that $\int_{\Sigma}u^{p^\ast}=1$. So thanks to Lemma \ref{key lemma} and to the first part of the proof of this theorem, we just have to set $g=\mathcal{U}_{1,1,\mathcal{O}}$ and prove that: 
\begin{center}
a nonegative function $u\in \dot{W}^{1,p}(\Sigma)$ such that $\Vert u\Vert_{L^{p^\ast}(\Sigma)}=\Vert \mathcal{U}_{1,1,\mathcal{O}}\Vert_{L^{p^\ast}(\Sigma)}(=1)$ achieves equality in \eqref{formula} if and only if there exist $a\in\mathbb{R}$, $\lambda>0$ and $x_0\in\overline{\Sigma}$ as in $(i)-(ii)-(iii)$ such that 
$$
u(x)=a\, \mathcal{U}_{1,1,\mathcal{O}}(\lambda(x-x_0))\, ,
$$
\end{center}

The main issue is that we have proved Lemma \ref{key lemma} (i.e.  \eqref{formula}) in the case when $f$ and $g$ are compactly supported, this restriction on $f$ and $g$ had no implication on the final inequality but it prevents to preserve the equality cases.

The idea is to proceed in two steps: 
\begin{itemize}
\item[1)] generalize the proof of \eqref{formula} in order to obtain the right inequality (i.e. for all admissible $f$ and $g$ not necessarily smooth and with compact support);
\item[2)] trace back all the equality cases in previous proof, without further assumptions on $f$ and $g$.
\end{itemize}
Observe that once this is done then the equality in the arithmetic-geometric inequality \eqref{AMGM} would imply that $\nabla\mathcal{T}$ is a point-wise multiple of the identity, from which it is easy to show that $\mathcal{T}(x)=\lambda(x-x_0)$,  for some $\lambda>0$ and $x_0\in\overline{\Sigma}$. 

The first step can be done as in \cite[Proof of Lemma 7]{CNV} by showing that the following inequality
$$
\int_{\Sigma} \mathrm{div}(\mathcal{T}(x))u^{p^\ast\left(1-\frac{1}{n}\right)}(x) \, dx\leq -\dfrac{p^\ast}{n}\left(1-\dfrac{1}{n}\right)\int_{\Sigma} u^{p^\ast\left(1-\frac{1}{n}\right)-1}(x) \mathcal{T}(x)\cdot\nabla u(x)\, dx \, ,
$$
holds true. This step can be performed by approximation and regularization and we refer to \cite[Proof of Lemma 7]{CNV} for the technical details.

Once we have this formula we prove that $u$ is positive.  This can be done arguing as in \cite[Step 1 of the proof of Proposition 6]{CNV} or arguing as in \cite[Appendix A]{CFR} by using the notion of indecomposability\footnote{A set of finite perimeter $E$ is said the be indecomposable if for every $F\subseteq E$ having finite perimeter and such that
$$
\mathrm{Per}(E)=\mathrm{Per}(F)+ \mathrm{Per}(E\setminus F)\, ,
$$
we have that
$$
\min\lbrace \vert E\vert,\vert E\setminus F\vert\rbrace=0\, . 
$$
This is a measure-theoretic notion similar to the topological notion of connectedness and we refer to \cite{ACMM} for more details.} of the support of $u$.  Here we sketch this second proof. Assume, by contradiction, that the support of $u$ is decomposable, i.e. one could write $u=u_1+u_2$, where $u_1$ and $u_2$ have disjoint supports. Then, of course
$$
\int_\Sigma \vert \nabla u(x)\vert^p\, dx=\int_\Sigma \vert \nabla u_1(x)\vert^p\, dx+\int_\Sigma \vert \nabla u_2(x)\vert^p\, dx \,  ,
$$ 
while from \eqref{Sob_sharp} and the fact that $u$ is an extremal, we would get 
\begin{align*}
\left(\int_\Sigma u(x)^{p^\ast}\, dx\right)^{\frac{p}{p^\ast}}&=\frac{1}{S_\Sigma^p}\int_\Sigma \vert \nabla u(x)\vert^p\, dx \\
&=\frac{1}{S_\Sigma^p}\int_\Sigma \vert \nabla u_1(x)\vert^p\, dx+\frac{1}{S_\Sigma^p}\int_\Sigma \vert \nabla u_2(x)\vert^p\, dx \\
&\geq \left(\int_\Sigma u_1(x)^{p^\ast}\, dx\right)^{\frac{p}{p^\ast}} + \left(\int_\Sigma u_2(x)^{p^\ast}\, dx\right)^{\frac{p}{p^\ast}} \, .
\end{align*}
Since $u_1$ and $u_2$ have disjoint support, then
$$
\int_\Sigma  u(x)^p\, dx=\int_\Sigma  u_1(x)^p\, dx+\int_\Sigma  u_2(x)^p\, dx \,  ,
$$ 
by the concavity of the function $t\mapsto t^{p/p^{\ast}}$ we conclude that either $u_1$ or $u_2$ vanishes. Hence we have that the support of $u$ is indecomposable.  Once we know this we can apply the classical interior regularity result in \cite{Caf3} for solutions of the Monge-Ampère equation to conclude that $\varphi\in W^{2,\alpha}_{loc}$. This implies that $\nabla\mathcal{T}$ has no singular part,\footnote{A different and more direct proof of this fact can be found in \cite[Step 2 in the proof of Proposition 6]{CNV}} hence the equality in the arithmetic-geometric inequality \eqref{AMGM} implies that the matrix $\nabla\mathcal{T}$ is a multiple of the identity which implies that $\mathcal{T}(x)=\lambda(x-x_0)$,  for some $\lambda>0$ and $x_0\in\overline{\Sigma}$ (see \cite[Step 3 in the proof of Proposition 6]{CNV}). Of course, from this fact the result follows easily, as already discussed.  Finally, properties $(i)-(ii)-(iii)$ on the location of $x_0$ follow for instance from the fact that $\mathcal{T}$ has to map $\Sigma$ onto $\Sigma$.

%
%

\end{proof}

\subsection{Weighted Sobolev inequalities:} actually the optimal transport approach is also suitable to prove a more general class of sharp weighted Sobolev inequalities.  Given $1<p<n$, a convex cone $\Sigma$ of $\mathbb{R}^n$ and given a weight $w\in C^0(\mathbb{R}^n)$ such that it is positive, homogeneous of degree $\alpha\geq 0$ and $w^{\frac{1}{\alpha}}$ is concave if $\alpha>0$, there exists a positive constant $S=S(n,p,a,w,H)$ such that
$$
 \left(\int_{\Sigma}H(\nabla u(x))^p w(x)\, dx\right)^{\frac{1}{p}}\geq S \left(\int_{\Sigma}\vert u(x)\vert^{\beta} w(x)\, dx\right)^{\frac{1}{\beta}}\, ,  \quad \text{ for all $u\in \dot{W}^{1,p}(\Sigma)$}\, ,
$$
where 
$$
\beta:=\frac{p(n+a)}{n+a-p}\, . 
$$
This results can be found in \cite{CROS,CFR,BGK} where the case of (more general) sharp weighted Sobolev inequalities in convex cones is treated.

\section{Critical points}\label{FR}

Besides the study of the extremals of the Sobolev inequality in $\mathbb{R}^n$ one interesting and challenging aspect is the study of critical points of the Sobolev inequality.  For simplicity and for later convenience we define the following 2-parameters subclass of extremals: 
\begin{equation}\label{Talentiane}
\mathcal{U}_{\lambda,x_0}(x):=\dfrac{\left[n\left(\frac{n-p}{p-1}\right)^{p-1}\lambda^p\right]^{\frac{n-p}{p^2}}}{\left(1+ \lambda^{\frac{p}{p-1}}\vert x-x_0\vert^{\frac{p}{p-1}} \right)^{\frac{n-p}{p}}}\, , 
\end{equation}
where $\lambda>0$ and $x_0\in\mathbb{R}^n$. We will refer to these kind of functions as \emph{Aubin-Talenti bubbles}. Roughly speaking the idea is the following: let $u$ be an Aubin-Talenti bubble and we compute the first variation of the Sobolev inequality (which is zero) in order to find the associated Euler-Lagrange equation. Explicitly, let 
$$
u(x)=\mathcal{U}_{\lambda,x_0}(x)\, , 
$$
and compute 
$$
\dfrac{d}{d\varepsilon}\left.\left(\Vert\nabla u+\varepsilon\nabla\varphi \Vert_{L^p(\mathbb{R}^n)}-S\Vert u+\varepsilon\varphi \Vert_{L^{p^\ast}(\mathbb{R}^n)}\right)\right|_{\varepsilon=0}=0\,  ,  \quad \text{for all $\varphi\in C^{\infty}_c(\mathbb{R}^n)$}\, .
$$
A direct computation yields to 
$$
\int_{\mathbb{R}^n}|\nabla u|^{p-2}\nabla u\cdot\nabla\varphi\, dx-\int_{\mathbb{R}^n}u^{p^\ast-1}\varphi\, dx=0 ,  ,  \quad \text{for all $\varphi\in C^{\infty}_c(\mathbb{R}^n)$}\, ,
$$
which is the weak formulation of the following quasilinear PDE
\begin{equation}\label{p-Laplace}
\Delta_p u+ u^{p^{\ast-1}}=0 \quad \text{in $\mathbb{R}^{n}$}\, ,
\end{equation}
which is called the \emph{critical $p$-Laplace equation} and where $\Delta_p u$ is the usual \emph{$p$-Laplace operator} defined in the following way
$$
\Delta_p u:=\mathrm{div}(\vert\nabla u\vert^{p-2}\nabla u)\, .
$$
Note that with the definition of Aubin-Talenti bubbles \eqref{Talentiane} we find that every Aubin-Talenti bubble solves \eqref{p-Laplace} exactly. Summing up,  we have shown that the Aubin-Talenti bubbles are solutions of the following quasilinear problem: 
\begin{equation}\label{EL}
		\begin{cases}
		\Delta_p u + u^{p^\ast-1} =0 & \text{ in } \mathbb{R}^n \\
		u>0  \,.
		\end{cases}
\end{equation}
The natual question now is: 
\begin{center}
are the Aubin-Talenti bubbles \eqref{Talentiane} the only solutions to \eqref{EL}? 
\end{center}
This question has attracted a lot of interest both in the PDE's and in the differential geometry communities. Indeed it is well-know that the critical Laplace equation (so $p=2$) is related to the \emph{Yamabe problem}.  Thanks to the efforts made in \cite{Yamabe, Trudinger,Obata,Aubin_bis,Schoen} (see also the survey \cite{LP}) we know the validity of the following theorem.

\begin{theorem}
Let $(M, g_0)$ be a compact Riemannian manifold of dimension $n \geq 3$. Then there exists a metric $g$ on $M$ which is conformal to $g_0$ and has constant scalar curvature.
\end{theorem}
If we write the conformal change from $g_0$ to $g$ in the following way
$$
g=u^{\frac{4}{n-2}}g_0
$$ 
for some positive and smooth function $u:M\rightarrow\mathbb{R}$. Then finding $g$ is equivalent to ask that $u$ solves the following PDE
$$
\frac{4(n-1)}{n-2}\Delta_{g_0}u-R_{g_0}u+R_{g}u^{\frac{n+2}{n-2}}=0\, ,
$$
where $R_{g_0}$ and $R_{g}$ denotes the scalar curvature of $M$ with respect to $g_0$ and $g$, respectively.  When $(M,g_0)$ is the round sphere, by stereographic projection we get that the previous PDE becomes
$$
\frac{4(n-1)}{n-2}\Delta u+R_{g}u^{\frac{n+2}{n-2}}=0 \quad \text{ in } \quad \mathbb{R}^n\, ,
$$
and hence (up to constants)
$$
\Delta u+u^{\frac{n+2}{n-2}}=0 \quad \text{in $\mathbb{R}^n$,}
$$
which is the critical Laplace equation.

Turning back to question of characterization of solutions to \eqref{EL} the state of the art is presented in the following theorem.

\begin{theorem}\label{Liouville}
Let $1<p<n$ and let $u$ be a solution to \eqref{EL} such that $u\in\dot{W}^{1,p}(\mathbb{R}^n)$. Then $u=\mathcal{U}_{\lambda,x_0}$, for some $ \lambda >0 $ and $ x_0 \in \mathbb{R}^n$.
\end{theorem}

In the semilinear case $p=2$ the theorem, without the assumption $u\in\dot{W}^{1,2}(\mathbb{R}^n)$, has been proved in the celebrated paper \cite{CGS} (see also \cite{CL} and \cite{Obata,GNN} for previous important results) by using the \emph{method of moving planes} (see the reviews \cite{Brezis} and \cite{CRBologna}) and the \emph{Kelvin transform}.  In the quasilinear case $p\neq 2$ the problem is more difficult because of the nonlinear structure of the $p$-Laplace operator and because of the lack of regularity of the solutions. Moreover, in the quaslilinear case the Kelvin transform is not available.  The first result related to the quasilinear case has been obtained in \cite{DMMS} for $\frac{2n}{n+2}\leq p<2$; the result has been extended to the case $1 <p <2$ in \cite{Vetois} and to the case $2<p<n$ in \cite{Sciunzi} exploiting a fine analysis of the behaviour of the solutions at infinity that allows to exploit the moving plane method as developed in \cite{DPR} and \cite{DR} (see also \cite{SZ}). We point out that the big difference between the semilinear and the quasilinear case is the additional assumption, in the quasilinear case, that $u$ has finite energy, i.e.  $u\in\dot{W}^{1,p}(\mathbb{R}^n)$: to prove the analogue result, removing this assumption, is an open and challenging problem! 

Finally, we mention that even in the semilinear case the assumption $u>0$ is fundamental, indeed it is possible to construct many sign-changing solutions to 
$$
\Delta u+ u\vert u\vert^{2^\ast-2}=0 \quad \text{in $\mathbb{R}^n$}\, ,
$$
which are not radial (see \cite{Ding} and also \cite{DMPP1,DMPP2,MW,MMW,MM}).  Moreover, we refer to positive solutions of \eqref{p-Laplace} since, from the maximum principle for quasilinear equations (see e.g. \cite{Vaz}), non-negative solutions to \eqref{p-Laplace} are either zero or positive.

The same procedure of computing the first variation of the Sobolev inequality can be done also starting from the anisotropic Sobolev inequality in convex cones \eqref{Sobolev_cone}. In analogy to \eqref{Talentiane} we consider the following subclass of extremals:
\begin{equation}\label{Talentiane_anis}
\mathcal{U}^{H}_{\lambda,x_0}(x):=
\frac{\left[n\left(\frac{n-p}{p-1}\right)^{p-1}\lambda^p\right]^{\frac{n-p}{p^2}}}{\left( 1 + \lambda^\frac{p}{p-1} H_0(x_0-x)^\frac{p}{p-1} \right)^{\frac{n-p}{p}}} \, , 
\end{equation}
where $\lambda>0$ and $x_0\in\overline{\Sigma}$ is such that if $\Sigma=\mathbb{R}^n$ then $x_0$ may be any point of $\mathbb{R}^n$, if $\Sigma=\mathcal{C}\times\mathbb{R}^k$ with $k\in \lbrace\, \dots,n-1\rbrace$ and $\mathcal{C}$ does not contain a line, then $x_0\in \mathbb{R}^k \times \lbrace \mathcal{O}\rbrace$, otherwise $x_0 = \mathcal{O}$.  We will refer to these kind of functions as anisotropic Aubin-Talenti bubbles.

As before, let
$$
u(x)=\mathcal{U}^{H}_{\lambda,x_0}(x)
$$
and computing as before the first variation of the anisotropic Sobolev inequality in convex cones we obtain 
$$
\int_{\Sigma} H^{p-1}(\nabla u)\nabla H(\nabla u)\cdot\nabla\varphi\, dx=\int_{\Sigma} u^{p^\ast-1}\varphi\, dx  ,  ,  \quad \text{for all $\varphi\in C^{\infty}(\Sigma)$}\, ,
$$
i.e. the weak formulation of the following Neumann quaslilinear problem: 
\begin{equation}\label{EL_bis}
		\begin{cases}
		\Delta^H_p u + u^{p^\ast-1} =0 & \text{ in } \Sigma \\
		a(\nabla u)\cdot\nu=0 & \text{ on } \partial\Sigma \,.
		\end{cases}
\end{equation}
where $\nu$ is the outward normal to $\partial\Sigma$,
$$
a(\nabla u):=H^{p-1}(\nabla u)\nabla H(\nabla u)
$$
and $\Delta^H_p u$ is the so-called \emph{anisotropic (or Finsler) $p$-Laplace operator} defined in the following way
$$
\Delta^H_p u:=\mathrm{div}(a(\nabla u))\, . 
$$
In \cite{CFR} we prove the following result which is the generalization of Theorem \ref{Liouville} for positive solutions to the problem \eqref{EL_bis}.

\begin{theorem}\label{thm}
Let $1<p<n$ and let $\Sigma=\mathbb{R}^k\times\mathcal{C}$ be a convex cone of $\mathbb{R}^n$, where $\mathcal C$ does not contain a line. Let $H$ be a norm of $\mathbb R^n$ such that $H^2$ is of class $C^{2}(\R^n\setminus \{\mathcal O\})\cap C^{1,1}(\mathbb{R}^n)$ and it is uniformly convex\footnote{i.e. there exist two constants $0<\lambda \leq \Lambda$ such that
\begin{equation*}
\lambda {\rm Id}\leq H(\xi)\,D^2H(\xi)+\nabla H(\xi)\otimes \nabla H(\xi) \leq \Lambda\,{\rm Id}\qquad \forall\,\xi \in \R^n\setminus \{\mathcal O\}
\end{equation*}
(note that $D^2(H^2)=2H\,D^2H+2\nabla H\otimes \nabla H$).}.
Let $u$ be a positive solution to \eqref{EL_bis} such that $u\in\dot{W}^{1,p}(\Sigma)$. Then $u(x)=\mathcal U_{\lambda,x_0}^H (x)$ for some $ \lambda >0 $ and $ x_0 \in \overline{\Sigma}$. Moreover, 
\begin{itemize}
	\item[$(i)$] if $k=n$ then $\Sigma = \mathbb{R}^n$ and $x_0$ may be a generic point in $\mathbb{R}^n$;
	\item[$(ii)$] if $k\in\{1,\dots,n-1\}$ then $x_0\in\mathbb{R}^k\times\mathcal{\{\mathcal O\}}$;
	\item[$(iii)$] if $k=0$ then $x_0=\mathcal{O}$.
\end{itemize} 
\end{theorem}

As already mentioned, case $(i)$ in Theorem \ref{thm} has been already proved in \cite{CGS,DMMS,Vetois,Sciunzi} when $\Sigma=\mathbb{R}^n$ and $H$ is the Euclidean norm.  Actually in the Euclidean case and when $p=2$, the classification result in convex cones was proved in \cite[Theorem 2.4]{LPT} by using the Kelvin transform (inspired by \cite{Gidas} and \cite{Obata}).

For general $1<p<n$, the Kelvin transform and the method of moving planes are not helpful neither for anisotropic problems nor inside cones. In \cite{CFR}, we provide a new approach to the characterization of solutions to critical $p-$Laplace equation, which is based on integral identities rather than the method of moving planes. This approach takes inspiration from \cite{SZ_bis}, where the authors prove nonexistence results generalizing the ones in \cite{Gidas} to $1<p<n$ and takes inspiration also from \cite{CS,BC,BCS,BNST,CR} where the authors prove symmetry and rigidity results for overdetermined problem (see \cite{Serrin, Weinberger,Reichel,PT}) in the anisotropic and in the conical settings.

Finally, we mention that also in this case the assumption $u>0$ is fundamental, indeed it is possible to construct sign-changing solutions to 
$$
		\begin{cases}
		\Delta u + \vert u\vert^{2^\ast-2}u =0 & \text{ in } \Sigma \\
		\partial_\nu u=0 & \text{ on } \partial\Sigma \,.
		\end{cases}
$$
which are non-radial (see \cite{CP}).

\section{Quantitative studies}\label{quant}

In this section we present some results for two important and fascinating problems related to the sharp Sobolev inequality in $\mathbb{R}^n$: the study of the stability of the extremals and of the critical points.

\subsection{Almost extremals.}
In this subsection we investigate another important aspect related to the sharp Sobolev inequality in $\mathbb{R}^n$: the stability of the Sobolev inequality \eqref{Sobolev} or the quantitative version of the Sobolev inequality \eqref{Sobolev}.  

Firstly, we indicate with $\mathcal{M}_p$ the $(n+2)-$dimensional manifold of all functions of the form \eqref{bubbles}, i.e. 
$$
\mathcal{M}_p:=\left\lbrace \mathcal{U}_{a,\lambda,x_0}(x):=\dfrac{a}{\left(1+ \lambda^{\frac{p}{p-1}}\vert x-x_0\vert^{\frac{p}{p-1}} \right)^{\frac{n-p}{p}}}  \, : \,  a\in\mathbb{R}\, , \lambda>0\,  , x_0\in\mathbb{R}^n \right\rbrace \, . 
$$
Secondly, we define the Sobolev deficit:
$$
\delta_p(u):=\frac{\Vert\nabla u \Vert_{L^p(\mathbb{R}^n)}}{\Vert u \Vert_{L^{p^\ast}(\mathbb{R}^n)}}-S\, ,  \quad \text{for all $u\in\dot{W}^{1,p}(\mathbb{R}^n)$}\, ;
$$
observe that, thanks to the sharp Sobolev inequality: 
\begin{center}
$\delta(u)\geq 0$ for all $u\in\dot{W}^{1,p}(\mathbb{R}^n)$ and $\delta(u)=0$ if and only if $u\in\mathcal{M}_p$.
\end{center}

The idea, based on a question in \cite{BL}, is the following: 
\begin{center}
for $p=2$ the Sobolev deficit can be estimated from below by some appropriate distance between $u$ and $\mathcal{M}_2$?
\end{center}

This problem has been solved in \cite{BE}, by showing that there exists a positive constant $c=c(n)$ such that 
$$
\delta_2(u)\geq c\inf_{v\in\mathcal{M}_2}\left(\frac{\Vert\nabla u-\nabla v\Vert_{L^2(\mathbb{R}^n)}} {\Vert\nabla u\Vert_{L^2(\mathbb{R}^n)}}\right)^{2} \, ,  \quad \text{for all $u\in\dot{W}^{1,2}(\mathbb{R}^n)$}\, ;
$$
and the result is optimal (in terms of the distance and in terms of the exponent $2$).

Now, the natural question is: 
\begin{center}
what about the general case, i.e. $1<p<n$?
\end{center}
The complete answer to this question has been recently provided in \cite{FZ} (see also \cite{CFMP,FN,N} for previous results) where the authors prove the following quantitative estimate: for $1<p<n$ there exists a positive constant $c=c(n,p)$ such that
$$
\delta_p(u)\geq c\inf_{v\in\mathcal{M}_p}\left(\frac{\Vert\nabla u-\nabla v\Vert_{L^p(\mathbb{R}^n)}} {\Vert\nabla u\Vert_{L^p(\mathbb{R}^n)}}\right)^{\alpha} \, ,  \quad \text{for all $u\in\dot{W}^{1,p}(\mathbb{R}^n)$}\, ,
$$
where the exponent $\alpha$ is given by $\max\lbrace 2,p\rbrace$ and it is optimal.
\subsection{Almost critical points.}
In this subsection we consider the Euler-Lagrange equation associated to the Sobolev inequality for $p=2$, i.e.  positive solutions to
\begin{equation}\label{2-Laplace}
\Delta u +  u^{2^\ast-1}=0 \quad \text{in $\mathbb{R}^n$}\, .
\end{equation}
In this subsection we want to investigate the following naïf question:
\begin{center}
if $u$ almost solves \eqref{2-Laplace}, then is it close to an Aubin-Talenti bubble?
\end{center}
In order to answer to this question we define the following deficit:
$$
\delta(u):=\Vert \Delta u + u^{2^\ast-1}\Vert_{H^{-1}}\, ,  \quad \text{for all $u\in\dot{W}^{1,2}(\mathbb{R}^n)$ and $u>0$}\, ;
$$
then it is clear, from Theorem \ref{Liouville}, that
\begin{center}
$\delta(u)=0$ if and only if $u$ is an Aubin-Talenti bubble \eqref{Talentiane}. 
\end{center}
Now, the question becomes the following: 
\begin{center}
if $\delta(u)$ is small, is the $u$ close to an Aubin-Talenti bubble?
\end{center}
The answer to this question is negative as one can see from the following example:
we set 
$$
u(x):=\mathcal{U}_{1,-Re_1}(x) + \mathcal{U}_{1,Re_1}(x)\, , \quad \text{for $R\gg 1$}\, . 
$$
In this case we say that $u$ is the sum of two weakly-interacting Aubin-Talenti bubbles. Then it is intuitive that $u$ will approximately solve  \eqref{2-Laplace} in any reasonable sense. But, of course, $u$ is not close to a single Aubin-Talenti bubble (in particular $\delta(u)\rightarrow 0$, as $R\rightarrow\infty$). Actually, this is the only possibility as shown in \cite{Struwe}, indeed the author shows the following
$$
\Gamma(u)\rightarrow 0\, ,  \quad \text{as $\delta(u)\rightarrow 0$}\, ,
$$
provided $u\in\dot{W}^{1,2}(\mathbb{R}^n)$ is such that
$$
\left(\nu-\frac{1}{2} \right)S^n\leq \int_{\mathbb{R}^n}\vert\nabla u\vert^2\, dx\leq \left(\nu-\frac{1}{2} \right)S^n \, , 
$$
and where  
$$
\Gamma(u):=\inf_{\lambda_i,x_i}||\nabla u-\sum_{i=1}^\nu\nabla\mathcal{U}_{\lambda_i,x_i}||_{L^2(\mathbb{R}^n)}
$$
denote the distance between $u$ and the sum of $\nu(\geq 1)$ Aubin-Talenti bubbles.

The quantitative version of the result in \cite{Struwe} has been the object of several studies, in particular we have the following estimates, according to the number of bubbles $\nu$ and to the dimension: 
$$
\Gamma(u) \lesssim \begin{cases}
	\delta(u) & \text{if } \nu=1, n\geq 3,   \\
	\delta(u) & \text{if } \nu>1,  3\leq n\leq 5, \\
	 \delta(u)\sqrt{\vert\log\delta(u)\vert}& \text{if } \nu>1,  n=6, \\ 
	\delta(u)^{\frac{n+2}{2(n-2)}} & \text{if } \nu>1, n\geq 7;
	\end{cases}  
$$
The first estimate can be found in \cite{CFM}, the second one in \cite{FG} and the third and the the last one in \cite{DSW} and we refer to the original papers for the complete statements and for comments.

\subsection{Further quantitative studies} Finally, we mention that, motivated by important applications in the calculus od variations and evolution PDEs,  the study of the quantitative stability of functional and geometric inequalities has been a growing interest in the recent years.  We refer to the survey papers \cite{Figalli,FigalliJerison,Fusco} for a general discussion and presentation of the results and we refer to \cite{FuMaPr,FMP,CiLe,CiSpa,FiZa,FiMa1,FiMa2,FiMaMo,FI} for the study of the stability for isoperimetric inequalities and to \cite{CiraoloMaggi, CV1,CV2,CVR,MP1,MP2,MP3,DMMN,KM} (see also the survey \cite{CiraoloUMI}) for the study of the stability of constant mean curvature hypersurfaces (i.e. the critical points of the classical isoperimetric inequality \eqref{iso} and these results are motivated by the celebrated Alexandrov soap bubbles theorem in \cite{Alex1} and \cite{Alex2}); moreover we refer to \cite{FiJe,FiMaPrBM,HiS,HST} for the study of the stability of the Brunn-Minkowski inequality, to \cite{CaFi,DoTo,Ng,Ruf} for the stability of the Gagliardo-Nirenberg inequality and to \cite{BWW,ChenFrank,Cianchi,DoZH,FMPlogSob,FMPJFA,GazzWeth,Carlen} (besides the already cited papers) for further stability results related to the Sobolev inequality (in the fractional case or for $p=1$).

%

\medskip

\subsection*{Acknowledgements} 
This survey is based on a short online talk that the author presented during the workshop ``Geometric theory of PDE's and sharp functional inequalities'' organized by Carlo Sbordone and Cristina Trombetti. The author has been the holder of a postdoc funded by the ``Istituto Nazionale di Alta Matematica" (INdAM). The author thanks Giulio Ciraolo for a careful reading of the manuscript.

\end{document}